\documentclass[11pt]{article}

\usepackage[utf8]{inputenc}
\usepackage[T1]{fontenc}
\usepackage{amsmath,amssymb,amsthm}
\usepackage[margin=1in]{geometry}
\usepackage{enumitem}
\usepackage{tikz}
\usetikzlibrary{calc}
\usepackage[colorlinks=true,linkcolor=blue,citecolor=blue,urlcolor=blue]{hyperref}
\usepackage{authblk}

\theoremstyle{plain}
\newtheorem{theorem}{Theorem}[section]
\newtheorem{lemma}[theorem]{Lemma}
\newtheorem{proposition}[theorem]{Proposition}
\newtheorem{corollary}[theorem]{Corollary}
\theoremstyle{definition}
\newtheorem{definition}[theorem]{Definition}
\newtheorem{example}[theorem]{Example}
\theoremstyle{remark}

\newcommand{\Sub}{\mathcal{S}}
\newcommand{\SubL}{\mathcal{S}_L}
\newcommand{\dT}{d_T}
\newcommand{\cost}{\operatorname{cost}}
\newcommand{\sd}{\mathbin{\triangle}}
\newcommand{\interior}[1]{#1^{\circ}}

\begin{document}

\title{A Crawling Metric on the Space of Subtrees of a Tree}

\author[1*]{Thu-Le Tran}
\author[2,3,4]{Cao Minh Quang}
\author[2,5]{Nguyen Chi Tam}
\author[1]{Kien Trung Nguyen}
\affil[1]{Faculty of Mathematics and Informatics, Teacher College, Can Tho University, Can Tho, Viet Nam}
\affil[2]{Faculty of Mathematics and Computer Science, University of Science, Ho Chi Minh City, Viet Nam}
\affil[3]{Vietnam National University, Ho Chi Minh City, Viet Nam}
\affil[4]{Nguyen Binh Khiem High School for the Gifted, Vinh Long, Viet Nam}
\affil[5]{Faculty of Information Technology, Cantho University of Technology, Can Tho, Viet Nam}
\affil[ ]{\textit{Emails:} \texttt{ttle@ctu.edu.vn} (T.-L.~Tran),
  \texttt{kt13quang@yahoo.com} (C.~M.~Quang),
  \texttt{nctam@ctuet.edu.vn} (N.~C.~Tam),
  \texttt{trungkien@ctu.edu.vn} (K.~T.~Nguyen)}
\affil[*]{Corresponding author}
\date{\today}

\maketitle

\begin{abstract}
We propose a crawling metric on the space of subtrees of a tree. The symmetric difference cannot tell whether two disjoint subtrees are close or far; we halve it and add the gap between them. The result is a metric, one half is the smallest factor, and the tree embeds canonically. The metric equals the least cost of crawling between subtrees by growing and shrinking without jumping. With simultaneous growth and shrinkage, subtrees of fixed length form a geodesic space.
\end{abstract}


\tableofcontents

\section{Introduction}\label{sec:intro}

Let $T$ be a tree whose edges carry positive lengths. In a \emph{subtree
location problem} the decision variable is not a point but a subtree $X$ of
$T$, typically of bounded total length, and the goal is to place $X$ well
with respect to weighted vertices $v_1,\dots,v_n$. The two classical
objectives are the \emph{median}, $\sum_j w_j\,\dT(v_j,X)$, and the
\emph{center}, $\max_j w_j\,\dT(v_j,X)$, together with their combinations
and generalizations \cite{HakimiSchmeichelLabbe1993,ShiouraUno1997,
TamirPuertoPerezBrito2002,PuertoTamir2005}, including robust versions with
uncertain vertex weights \cite{RobustMedianSubtree} and inverse versions in
which the tree is modified so that a given subtree becomes optimal
\cite{InverseCenterSubtree}. Several of these problems are solved by
greedy or local rules, which behave as if the objective were convex. Since
the variable ranges over a family of sets rather than a convex region of
$\mathbb{R}^n$, explaining this behavior first requires a geometry on the
space of subtrees: a distance, and a notion of how one subtree moves to
another. This paper studies the distance between subtrees.

When all subtrees contain a common point, a natural choice is the length of
the symmetric difference, $\mu(X\sd Y)$. A subtree is then moved by
shrinking the part not shared and growing the missing part, and the cost of
this move is exactly $\mu(X\sd Y)$. Without a common point this distance
fails. Two disjoint subtrees of length $L$ are always at distance $2L$,
whether they touch at their ends or sit at opposite sides of $T$. The
symmetric difference sees which points are covered, but not how far a
subtree must travel.

This note proposes a distance that records travel. For subtrees $X,Y$ put
\[
  d(X,Y)=\tfrac12\,\mu(X\sd Y)+\dT(X,Y),
\]
where $\dT(X,Y)$ is the length of the shortest arc joining $X$ and $Y$, and
is $0$ when they meet. The first term pays for what is removed and added;
the second pays for the way between them.

This distance generally differs from the Hausdorff distance.
Up to a factor of $\tfrac12$, it has the algebraic form of
pseudometrics induced by isotone valuations on
join-semilattices~\cite{ramana1985valuation}.
Indeed, writing $X\vee Y$ for the smallest subtree containing
$X\cup Y$, we have
\[
  d(X,Y)=\mu(X\vee Y)-\tfrac12\bigl(\mu(X)+\mu(Y)\bigr).
\]
In the present setting, we prove that $d$ is a metric on the
space of all subtrees.

The contributions are as
follows.
\begin{enumerate}[label=(\roman*)]
  \item $d$ is a metric on the space of all subtrees. The factor
    $\tfrac12$ is the smallest coefficient of $\mu(X\sd Y)$ for which this
    holds, and the tree $(T,\dT)$ embeds isometrically in this space as its
    single points (Section~\ref{sec:metric}).
  \item $d$ is the least cost of a \emph{crawl}, a sequence of overlapping
    subtrees leading from $X$ to $Y$. Two crawls attain this cost: one
    inflates to the join $X\vee Y$ and then
    deflates; the other, when $X$ and $Y$ meet, deflates to $X\cap Y$ and
    then inflates (Section~\ref{sec:crawl}).
  \item On the subtrees of a fixed length, growth and shrinkage can be
    performed simultaneously, which yields a geodesic between any two such
    subtrees that never changes length. Hence this family is a geodesic
    space (Section~\ref{sec:geodesic}).
\end{enumerate}

The paper is organized as follows. Section~\ref{sec:metric} defines $d$ and
proves that it is a metric. Section~\ref{sec:crawl} interprets $d$ as a
least crawling cost. Section~\ref{sec:geodesic} constructs geodesics among
subtrees of a fixed length. Section~\ref{sec:conclusion} concludes.

\section{A metric on the space of subtrees}\label{sec:metric}

Throughout, $T$ is a finite tree whose edges are segments of positive
length, viewed as a metric space $(T,\dT)$ in which $\dT(x,y)$ is the length
of the unique arc $[x,y]$. We assume $T$ has at least one edge. The length
measure is denoted $\mu$. A \emph{subtree} is a nonempty, closed, connected
subset of $T$; single points are subtrees of length $0$. We write $\Sub$ for
the set of subtrees. For $X,Y\in\Sub$ we put
$\dT(X,Y)=\min\{\dT(x,y):x\in X,\ y\in Y\}$; the minimum exists by
compactness, and $\dT(X,Y)=0$ if and only if $X\cap Y\neq\emptyset$.

A subtree may cover some edges completely and others only in part, and its
length adds up these pieces. Figure~\ref{fig:length} gives an example.

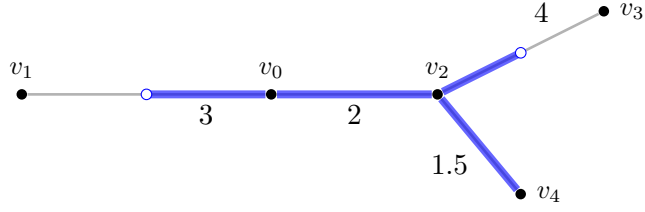
\begin{figure}[h]
\centering
\begin{tikzpicture}[scale=1.1,
  vtx/.style={circle,fill=black,inner sep=1.4pt},
  every label/.style={font=\small}]
  \node[vtx,label=above:$v_0$] (v0) at (0,0) {};
  \node[vtx,label=above:$v_1$] (v1) at (-3,0) {};
  \node[vtx,label=above:$v_2$] (v2) at (2,0) {};
  \node[vtx,label=right:$v_3$] (v3) at (4,1) {};
  \node[vtx,label=right:$v_4$] (v4) at (3,-1.2) {};
  \draw[gray!60,line width=1pt] (v0)--(v1) node[pos=0.25,below,black]{$3$};
  \draw[gray!60,line width=1pt] (v0)--(v2) node[midway,below,black]{$2$};
  \draw[gray!60,line width=1pt] (v2)--(v3) node[pos=0.75,above left,black]{$4$};
  \draw[gray!60,line width=1pt] (v2)--(v4) node[midway,below left,black]{$1.5$};
  \coordinate (a) at ($(v0)!0.5!(v1)$);
  \coordinate (b) at ($(v2)!0.5!(v3)$);
  \draw[blue,line width=3pt,opacity=0.6] (a)--(v0)--(v2)--(b);
  \draw[blue,line width=3pt,opacity=0.6] (v2)--(v4);
  \node[circle,draw=blue,fill=white,inner sep=1.4pt] at (a) {};
  \node[circle,draw=blue,fill=white,inner sep=1.4pt] at (b) {};
\end{tikzpicture}
\caption{A tree with edge lengths $3,2,4,1.5$ and a subtree $X$ (thick).
$X$ covers half of $[v_0,v_1]$, all of $[v_0,v_2]$, half of $[v_2,v_3]$
and all of $[v_2,v_4]$, so $\mu(X)=1.5+2+2+1.5=7$.}
\label{fig:length}
\end{figure}

When two subtrees are disjoint, they are joined by a unique shortest arc.

\begin{lemma}[Bridge]\label{lem:bridge}
Let $X,Y\in\Sub$ with $X\cap Y=\emptyset$. There are unique points $p\in X$, $q\in Y$
such that the arc $P=[p,q]$ satisfies $P\cap X=\{p\}$ and $P\cap Y=\{q\}$.
Every arc from a point of $X$ to a point of $Y$ contains $P$, and
$\dT(X,Y)=\mu(P)>0$.
\end{lemma}

\begin{proof}
Fix $x\in X$, $y\in Y$. Let $q$ be the first point of $[x,y]$, starting
from $x$, that lies in $Y$, and let $p$ be the last point of $[x,q]$ lying
in $X$. Then $p\neq q$ and the open arc $\interior{P}=(p,q)$ misses
$X\cup Y$. Removing a point $b\in\interior{P}$ separates $x$ from $y$, and
since $X$ and $Y$ are connected and avoid $b$, it separates $X$ from $Y$.
Hence every arc from $X$ to $Y$ passes through $b$, and therefore contains
$\interior{P}$ and its closure $P$. Such an arc has length at least
$\mu(P)$, with equality only for $P$ itself.
\end{proof}

We call $P$ the \emph{bridge} between $X$ and $Y$ and write
$\delta=\mu(P)=\dT(X,Y)$. When $X\cap Y\ne\emptyset$ we set $P=\emptyset$
and $\delta=0$. The smallest subtree containing $X\cup Y$ is
\[
  X\vee Y=X\cup P\cup Y,
\]
which equals $X\cup Y$ when $X\cap Y\ne\emptyset$; it is a subtree since
$P$ meets $X$ at $p$ and $Y$ at $q$. Since $P$ meets $X\cup Y$ only at its
endpoints,
\begin{equation}\label{eq:join-length}
  \mu(X\vee Y)=\mu(X\cup Y)+\dT(X,Y).
\end{equation}

\begin{definition}[Crawling metric]\label{def:d}
For $X,Y\in\Sub$,
\[
  d(X,Y)=\tfrac12\,\mu(X\sd Y)+\dT(X,Y).
\]
\end{definition}

\begin{example}\label{ex:basic}
On a segment, let $X=[0,L]$ and $Y=[L+\delta,2L+\delta]$. Then
$d(X,Y)=L+\delta$, which is the distance by which $X$ must be translated to
reach $Y$. For $X=[0,L]$ and $Y=[s,L+s]$ with $0\le s\le L$ one gets
$d(X,Y)=s$, again the translation length.
\end{example}

The only nontrivial axiom is the triangle inequality when the outer two
subtrees are disjoint. It rests on the following covering property of the
bridge.

\begin{lemma}[Bridge covering]\label{lem:cover}
Let $X,Y,Z\in\Sub$ with $X\cap Z=\emptyset$, and let $P$ be the bridge
between $X$ and $Z$. Let $P_{XY}$ and $P_{YZ}$ denote the bridges between
$X,Y$ and between $Y,Z$ (empty if the sets meet). Then
\[
  \interior{P}\subseteq Y\cup P_{XY}\cup P_{YZ}.
\]
\end{lemma}

\begin{proof}
Let $b\in\interior{P}\setminus Y$. Removing $b$ separates $X$ from $Z$, and
the connected set $Y$ lies on one side. If $Y$ lies on the side not
containing $X$, then $Y\cap X=\emptyset$ and every arc from $X$ to $Y$
passes through $b$, so $b\in P_{XY}$. Otherwise $Y$ lies on the side of
$X$, away from $Z$, and likewise $b\in P_{YZ}$.
\end{proof}

\begin{theorem}\label{thm:metric}
$d$ is a metric on $\Sub$.
\end{theorem}

\begin{proof}
Nonnegativity and symmetry are clear.

\emph{Identity.} Let $d(X,Y)=0$. Then $X\cap Y\neq\emptyset$ and
$\mu(X\sd Y)=0$. Suppose $x\in X\setminus Y$. Since $Y$ is closed, some
ball $B$ around $x$ misses $Y$. As $X$ meets $Y$, $X\ne\{x\}$, so $X$
contains a nondegenerate arc starting at $x$, whose initial piece lies in
$B\cap X\subseteq X\setminus Y$ and has positive length, a contradiction.
Hence $X\subseteq Y$, and by symmetry $X=Y$.

\emph{Triangle inequality.} Let $X,Y,Z\in\Sub$. Since
$X\sd Z\subseteq(X\sd Y)\cup(Y\sd Z)$,
\begin{equation}\label{eq:sd-triangle}
  \mu(X\sd Z)\le\mu(X\sd Y)+\mu(Y\sd Z).
\end{equation}
If $X\cap Z\ne\emptyset$, then $\dT(X,Z)=0$ and \eqref{eq:sd-triangle}
gives the claim. Assume $X\cap Z=\emptyset$ and let $P$ be their bridge,
$c=\mu(P)$. Then $\mu(X\sd Z)=\mu(X)+\mu(Z)$, and expanding the symmetric
differences,
\[
  \tfrac12\bigl[\mu(X\sd Y)+\mu(Y\sd Z)\bigr]
  =\tfrac12\bigl[\mu(X)+\mu(Z)\bigr]+\mu\bigl(Y\setminus(X\cup Z)\bigr),
\]
where we used $X\cap Z=\emptyset$. The triangle inequality is therefore
equivalent to
\begin{equation}\label{eq:key}
  c\le\mu\bigl(Y\setminus(X\cup Z)\bigr)+\dT(X,Y)+\dT(Y,Z).
\end{equation}
By Lemma~\ref{lem:cover},
$c=\mu(\interior{P})\le\mu(Y\cap\interior{P})+\mu(P_{XY})+\mu(P_{YZ})$.
Since $\interior{P}$ misses $X\cup Z$, we have
$Y\cap\interior{P}\subseteq Y\setminus(X\cup Z)$, and
$\mu(P_{XY})=\dT(X,Y)$, $\mu(P_{YZ})=\dT(Y,Z)$. This is
\eqref{eq:key}.
\end{proof}

We now show that the factor $\tfrac12$ in Definition~\ref{def:d} cannot be
lowered.

\begin{proposition}\label{prop:half}
For $\lambda>0$ let $d_\lambda(X,Y)=\lambda\,\mu(X\sd Y)+\dT(X,Y)$. Then
$d_\lambda$ is a metric on $\Sub$ if and only if $\lambda\ge\frac12$.
\end{proposition}

\begin{proof}
If $\lambda\ge\frac12$, then
$d_\lambda=d+(\lambda-\frac12)\,\mu(\cdot\sd\cdot)$ is the sum of a metric
and a pseudometric, hence a metric. If $\lambda<\frac12$, take an edge
$[a,b]$ of length $c>0$ and put $X=\{a\}$, $Y=[a,b]$, $Z=\{b\}$. Then
$d_\lambda(X,Z)=c>2\lambda c=d_\lambda(X,Y)+d_\lambda(Y,Z)$.
\end{proof}

For $\lambda=\frac12$ the example gives equality: moving a point by first
growing it into a segment and then shrinking the segment to its other end
costs exactly the distance travelled. This is the mechanism studied in
Section~\ref{sec:crawl}.

\begin{proposition}\label{prop:embed}
The map $x\mapsto\{x\}$ is an isometric embedding of $(T,\dT)$ into
$(\Sub,d)$.
\end{proposition}

\begin{proof}
For $x,y\in T$, $\mu(\{x\}\sd\{y\})=0$, so
$d(\{x\},\{y\})=\dT(\{x\},\{y\})=\dT(x,y)$. Injectivity follows.
\end{proof}

Thus $(\Sub,d)$ contains the tree itself as the subspace of single points.
The symmetric difference metric $\mu(X\sd Y)$ does not have this property since 
it vanishes on all pairs of points.

\section{Crawling interpretation}\label{sec:crawl}

Section~\ref{sec:metric} defined $d$ by a formula. This section shows that
the formula is forced: $d$ is the least cost of moving a subtree through
overlapping positions, without jumps.

\begin{definition}[Crawl]\label{def:crawl}
A \emph{crawl} from $X$ to $Y$ is a finite sequence
$\gamma:\ X=Z_0,Z_1,\dots,Z_k=Y$ of subtrees in which consecutive subtrees
overlap, $Z_{i-1}\cap Z_i\ne\emptyset$. Its cost is
\[
  \cost(\gamma)=\sum_{i=1}^k d(Z_{i-1},Z_i)
  =\tfrac12\sum_{i=1}^k\mu(Z_{i-1}\sd Z_i).
\]
\end{definition}

The overlap condition forbids jumps: each step keeps part of the current
subtree and pays only for what is added and removed. Consecutive subtrees
need not be nested; a step may add on one side and remove on another. A
segment of length $L$ shifted by $\varepsilon\le L$ adds $\varepsilon$ at
its front and removes $\varepsilon$ at its back, for a cost $\varepsilon$.

The following representation expresses $d$ through the length of the join
$X\vee Y$.

\begin{proposition}\label{prop:hull-distance}
For all $X,Y\in\Sub$,
\[
  d(X,Y)=\mu(X\vee Y)-\tfrac12\bigl(\mu(X)+\mu(Y)\bigr).
\]
\end{proposition}

\begin{proof}
By \eqref{eq:join-length} and
$\mu(X\cup Y)-\frac12(\mu(X)+\mu(Y))=\frac12\mu(X\sd Y)$,
\[
  \mu(X\vee Y)-\tfrac12\bigl(\mu(X)+\mu(Y)\bigr)
  =\tfrac12\mu(X\sd Y)+\dT(X,Y)=d(X,Y).\qedhere
\]
\end{proof}

This representation immediately yields an optimal crawl: first grow to the
join, then shrink to the target.

\begin{proposition}[Inflate, then deflate]\label{prop:crawl1}
For all $X,Y\in\Sub$, the crawl $\gamma:\ X\to X\vee Y\to Y$ satisfies
$\cost(\gamma)=d(X,Y)$.
\end{proposition}

\begin{proof}
Both steps are nested, hence overlapping. By
Proposition~\ref{prop:hull-distance}, the cost is
\[
  \tfrac12\mu\bigl((X\vee Y)\setminus X\bigr)
  +\tfrac12\mu\bigl((X\vee Y)\setminus Y\bigr)
  =\mu(X\vee Y)-\tfrac12\bigl(\mu(X)+\mu(Y)\bigr)=d(X,Y).\qedhere
\]
\end{proof}

\begin{proposition}[Deflate, then inflate]\label{prop:crawl2}
If $X\cap Y\ne\emptyset$, the crawl $\gamma:\ X\to X\cap Y\to Y$ satisfies
$\cost(\gamma)=d(X,Y)$.
\end{proposition}

\begin{proof}
$X\cap Y$ is a subtree, and both steps are nested. The cost is
$\tfrac12\mu(X\setminus Y)+\tfrac12\mu(Y\setminus X)=\tfrac12\mu(X\sd Y)
=d(X,Y)$.
\end{proof}

The first crawl passes through $X\vee Y$, whose length may reach
$\mu(X)+\mu(Y)+\delta$; the second passes through $X\cap Y$, which may be a
single point. Both extremes are optimal.

\begin{theorem}\label{thm:rho}
For all $X,Y\in\Sub$,
\[
  d(X,Y)=\min_\gamma\cost(\gamma),
\]
the minimum over all crawls $\gamma$ from $X$ to $Y$.
\end{theorem}

\begin{proof}
For any crawl, the triangle inequality of Theorem~\ref{thm:metric} gives
$\cost(\gamma)=\sum_i d(Z_{i-1},Z_i)\ge d(X,Y)$.
Proposition~\ref{prop:crawl1} attains this bound.
\end{proof}

\section{Geodesics among subtrees of fixed length}\label{sec:geodesic}

For $L>0$ let
\[
  \SubL=\{X\in\Sub:\ \mu(X)=L\},
\]
which is nonempty when $L\le\mu(T)$. On $\SubL$,
$\mu(X\setminus Y)=\mu(Y\setminus X)$, so
$d(X,Y)=\mu(X\setminus Y)+\dT(X,Y)$. The crawls of
Section~\ref{sec:crawl} need not remain in $\SubL$. Here growth and
shrinkage occur simultaneously, which preserves length exactly and gives a
constant-speed geodesic inside $\SubL$.

Let $X,Y\in\SubL$ and put
\[
  m=\mu(X\vee Y)-L=d(X,Y),
\]
by Proposition~\ref{prop:hull-distance}. If $m=0$ then $X=Y$, so assume
$m>0$. For $r,s\ge0$ consider
\[
  G(r)=\{x\in X\vee Y:\ \dT(x,X)\le r\},\qquad
  S(s)=\{x\in X\vee Y:\ \dT(x,Y)\le s\}.
\]
Both are subtrees, since every point can be joined to $X$ (resp.\ $Y$) by
an arc inside the set. Their lengths are continuous and monotone in the
radius, from $L$ to $\mu(X\vee Y)=L+m$. For $a\in[0,m]$ choose radii
$r(a)$ nondecreasing and $s(a)$ nonincreasing, with $r(0)=0$ and $s(m)=0$,
such that
\[
  \mu\bigl(G(r(a))\bigr)=L+a,\qquad \mu\bigl(S(s(a))\bigr)=L+m-a,
\]
and write
\[
  G_a=G(r(a)),\qquad S_a=S(s(a)),\qquad R_a=(X\vee Y)\setminus S_a,\qquad
  F_a=G_a\cap S_a=G_a\setminus R_a.
\]
Thus $G_a$ is what has grown from $X$ and $R_a$ is what has been removed
towards $Y$, with $\mu(G_a\setminus X)=\mu(R_a)=a$. The continuous crawl
from $X$ to $Y$ is
\[
  H(t)=F_{mt}\qquad(0\le t\le1).
\]

\begin{theorem}\label{thm:geodesic}
$H$ maps $[0,1]$ into $\SubL$, $H(0)=X$, $H(1)=Y$, and
\[
  d(H(s),H(t))=|t-s|\,d(X,Y)\qquad(s,t\in[0,1]).
\]
\end{theorem}

\begin{proof}
\emph{Nothing is removed before it is reached: $R_a\subseteq G_a$.} Let
$x\in R_a$. Then $x\notin Y$; if $x\in X$ we are done. Otherwise $X,Y$ are
disjoint and $x$ lies inside their bridge $[p,q]$, at distance $u$ from
$p$. Every point of $[p,x]$ is at least as far from $Y$ as $x$, so
$[p,x]\subseteq R_a$ and $a\ge u$. If $x\notin G_a$, then $r(a)<u<\delta$,
and growth from $X$ within $X\vee Y$ by a radius below $\delta$ covers only
the initial segment of the bridge of length $r(a)$; thus $a=r(a)<u$, a
contradiction.

\emph{Length and endpoints.} Hence $\mu(F_a)=\mu(G_a)-\mu(R_a)=L$. As
$L>0$, $F_a$ is a nonempty intersection of two subtrees, hence a subtree,
so $F_a\in\SubL$. Moreover $F_0=X\cap(X\vee Y)=X$ and
$F_m=(X\vee Y)\cap Y=Y$.

\emph{Lipschitz bound.} For $a<b$, nesting gives
$F_a\sd F_b\subseteq(G_b\setminus G_a)\cup(R_b\setminus R_a)$, so
$\mu(F_a\sd F_b)\le2(b-a)$. If $b-a<L$, then $F_a$ and $F_b$ meet, since
disjoint members of $\SubL$ have symmetric difference $2L$; hence
$d(F_a,F_b)=\frac12\mu(F_a\sd F_b)\le b-a$. Splitting $[a,b]$ into pieces
shorter than $L$ and using the triangle inequality gives
$d(F_a,F_b)\le b-a$ for all $a<b$.

\emph{Equality.} For $0\le s<t\le1$,
\[
  m=d(X,Y)\le d(X,H(s))+d(H(s),H(t))+d(H(t),Y).
\]
Thus, 
\[
  m
  \le ms+m(t-s)+m(1-t)=m,
\]
so every inequality is an equality and $d(H(s),H(t))=m(t-s)$.
\end{proof}

\begin{corollary}\label{cor:geodesic}
For every $0<L\le\mu(T)$, the metric space $(\SubL,d)$ is geodesic.
\end{corollary}

Unlike the crawls of Section~\ref{sec:crawl}, which separate growth from
shrinkage, the continuous crawl $H$ adds at the front exactly what it
removes at the back. Nearby positions $H(s)$ and $H(t)$ overlap but are not
nested, which is why Definition~\ref{def:crawl} requires overlap rather
than inclusion.

\section{Conclusion}\label{sec:conclusion}

We introduced a distance between subtrees of a tree that adds the gap
between two subtrees to half of their symmetric difference. It is a metric
on all subtrees, the factor one half is the smallest possible, and the tree
embeds in the resulting space as its single points. The metric has a
mechanical meaning: it is the least cost of crawling from one subtree to
another through overlapping positions. This least cost is reached both by
inflating then deflating and, for overlapping subtrees, by deflating then
inflating. When growth and shrinkage are synchronized, a subtree moves
without changing its length, so the subtrees of a fixed length form a
geodesic space.

The metric provides a geometric framework for subtree location problems.
Two directions merit further study. The first is to characterize the
geodesics between two subtrees and determine how to select among
different orders of growth and shrinkage. The second is to investigate
whether this metric and geodesic structure can explain existing
efficient algorithms for subtree median and center problems and
guide the development of new ones.

\bibliographystyle{plain}
\bibliography{main}

\begin{thebibliography}{1}

\bibitem{HakimiSchmeichelLabbe1993}
S.~Louis Hakimi, Edward~F. Schmeichel, and Martine Labb{\'e}.
\newblock On locating path- or tree-shaped facilities on networks.
\newblock {\em Networks}, 23(6):543--555, 1993.

\bibitem{RobustMedianSubtree}
Kien~Trung Nguyen, Huong Nguyen-Thu, Huy~Minh Le, and Tran~Thanh Hiep.
\newblock Combinatorial algorithms for the minmax robust median subtree problem
  on tree networks with interval vertex weights.
\newblock {\em Asian-European Journal of Mathematics}, 18(02):2450107, 2025.

\bibitem{InverseCenterSubtree}
Kien~Trung Nguyen, Huong Nguyen-Thu, and Nguyen~Thanh Luan.
\newblock The inverse center subtree problem on tree graphs.
\newblock {\em Discrete Mathematics, Algorithms and Applications},
  18(02):2550030, 2026.

\bibitem{PuertoTamir2005}
Justo Puerto and Arie Tamir.
\newblock Locating tree-shaped facilities using the ordered median objective.
\newblock {\em Mathematical Programming}, 102(2):313--338, 2005.

\bibitem{ramana1985valuation}
PV~Ramana~Murty and Teresa Engelbert.
\newblock On valuation in semilattices.
\newblock {\em Math{\'e}matiques et sciences humaines}, 90:19--44, 1985.

\bibitem{ShiouraUno1997}
Akiyoshi Shioura and Takeaki Uno.
\newblock A linear time algorithm for finding a $k$-tree core.
\newblock {\em Journal of Algorithms}, 23(2):281--290, 1997.

\bibitem{TamirPuertoPerezBrito2002}
Arie Tamir, Justo Puerto, and Dionisio P{\'e}rez-Brito.
\newblock The centdian subtree on tree networks.
\newblock {\em Discrete Applied Mathematics}, 118(3):263--278, 2002.

\end{thebibliography}

\end{document}